\documentclass{article}
\usepackage{amsmath}
\usepackage{amsthm,amsfonts,amssymb,amscd,euscript}
\usepackage{graphicx}
\usepackage{color}
\usepackage[colorlinks]{hyperref}
\usepackage{tikz}
\usepackage{
soul}
\usetikzlibrary{arrows}
\usepackage{wrapfig}
\usepackage{mathdots}
\usepackage{yhmath}
\usepackage{cancel}
\usepackage{siunitx}
\usepackage{array}
\usepackage{multirow}
\usepackage{tabularx}
\usepackage{booktabs}
\usepackage{wrapfig}
\usepackage{subcaption}
\usepackage{textcomp, gensymb}
\usepackage{enumitem}
\usepackage{ upgreek }
\usepackage[all]{xy}
\usepackage{mathtools}
\usepackage[normalem]{ulem}
\newtheorem{thm}{Theorem}[section]
\newtheorem{cor}[thm]{Corollary}
\newtheorem{lem}[thm]{Lemma}

\newtheorem{prop}[thm]{Proposition}
\newtheorem{example}[thm]{Example}
\newtheorem{defn}[thm]{Definition}
\theoremstyle{remark}
\newtheorem{rem}[thm]{Remark}

\newcommand{\ah}[1]{{\color{red!50!blue} #1}}

\newcommand{\Z}{\mathbb Z}
\newcommand{\N}{\mathbb N}
\newcommand{\IFS}{$\mathfrak{I}=(X,\,\mathcal{F},\,\Sigma)$\ }

\makeatother

\numberwithin{equation}{section}

\begin{document}
	\title{Attractor and its self-similarities for an IFS over arbitrary sub-shift}
	\author{Dawoud AHMADI DASTJERDI\footnote{Corressponding author, dahmadi1387@gmail.com}, Sedigheh DARSARAEE
	}

	\maketitle	
	\date{}	
	\begin{abstract}
Consider a compact metric space $X$, and let $\mathcal{F}=\{f_1,\,f_2,\ldots,\, f_k\}$ be a set of contracting and continuous self maps on $X$. Let $\Sigma$ be a sub-shift on $k$ symbols, and let $\Sigma_k$ be the full shift.
 Define $\mathcal{L}_n(\Sigma)$ as the set of words of length $n$ in $\Sigma$. For $u=u_1\cdots u_n\in \mathcal{L}_n(\Sigma)$, set $f_u:=f_{u_n}\circ\cdots \circ f_{u_1}$ and $H^n(\cdot):=\cup_{u \in \mathcal{L}_{n}(\Sigma_k)} f_{u}(\cdot)$.
  When $\Sigma=\Sigma_k$, $H^n(\cdot)$ is the $n$th iteration of the Hutchinson's operator, and there exists a compact set 
 $S= \lim_{n \rightarrow \infty} H^n(A)$ for any compact $A\subseteq X$ with $H^n(S)=S$ (self-similarity criteria) for $n\in\N$.
 For arbitrary $\Sigma$, the above limit exists; but
   it is not necessarily true that $H^n(S)=S$.
Sufficient conditions on $\Sigma$ are provided to have $H^n(S)=S$ for all or some $n\in\N$, and then the dynamics of $S$ under the admissible iterations of $f_i$'s defined by $\Sigma$ are investigated.

	\end{abstract}

\section{Introduction}
Let $X$ be a compact metric space and for $1\leq  i\leq k$, $f_i:X\rightarrow X$ a continuous map. 
Let $\Sigma_k$ be the full shift on $k$ symbols and let $\mathcal{L}(\Sigma_k)=\{u: u=u_1\cdots u_n\text{ is a word in } \Sigma_k, \,n\in\N\}$ its language. Then under the concatenation of words, 
$\mathcal{L}(\Sigma_k)$ is a semi-group  and the action of 
$\mathcal{L}(\Sigma_k)$ on $X$ defined by 
\begin{equation}\label{eq combination of maps}
(u_1\cdots u_n,\,x)\mapsto f_u(x):=f_{u_n}\circ\cdots \circ f_{u_1}(x)
\end{equation}
 is a topological dynamical system called a  \emph{``conventional" iterated function system} (IFS).
The IFS was first introduced by M. Barnsley and S. Demko for the scenario where $f_i$'s were contractions and $\Sigma$ a full shift on $k$ symbols \cite{barnsley1985iterated}. Building upon Hutchinson's earlier findings \cite{hutchinson1981fractals}, their goal was to devise a technique for generating fractals, or sets with non-integer Hausdorff dimensions. This concept gained prominence through the work of Peitgen and Richter \cite{peitgen1986beauty} and Barnsley et al. \cite{barnsley1988fractals}; notably, they produced digital images that evoke natural phenomena.
Now, many researchers believe that IFS has the potential to give  mathematical models for some elements in nature. See for instance \cite{dokukin2011cell} and \cite{leggett2019motility}.
These beliefs stem from the fact that chance plays a major role in nature and an IFS.
So it is reasonable to study IFS from different viewpoints;
namely, probabilistic, geometric-measure theory, and topological.
Here, we will focus primarily on topological aspects and defer the study of ergodic behavior for future research.

For us, an IFS is a triple \IFS where $\mathcal{F}$ is the
 set of above $k$ self-maps 
 that, unless otherwise specified, are contractions, and
  $\Sigma$ is an arbitrary sub-shift on $k$ symbols satisfying \eqref{eq combination of maps}. This seems natural; as one would anticipate that in general, certain sequences of map iterations are forbidden, so leading to a generic sub-shift.
 Note that under this framework, $\mathfrak{I}=(X,\,\mathcal{F},\,\Sigma_k)$ represents the conventional IFS.

For a conventional IFS, Hutchinson utilized the Hausdorff metrics on 
$\mathcal{C}=\{A\subseteq X:\; A \text{ is compact}\}$
 and introduced \emph{Hutchinson operator} $H:\mathcal{C}\rightarrow\mathcal{C}$ defined as
$H(A):=\bigcup_{a=1}^kf_a(A)$.
Consequently,
\begin{equation}\label{eq H^n(A)}
H^n(A)=\bigcup_{u\in\mathcal{L}_n(\Sigma_k)}f_u(A)
\end{equation}
holds true, where
 $\mathcal{L}_n(\Sigma_k)\subset \mathcal{L}(\Sigma_k)$ denotes the collection of words with length $n$ and $H^n(\cdot)$ is the $n$th iteration of $H$.
Subsequently, it was established that there exists  $S\in\mathcal{C}$ such that $\lim_n H^n(A)=S$,
 identifying this distinctive compact set as the \emph{attractor} of the IFS.
 Additionally, the relation
 \begin{equation}\label{eq self-similarity}
	H^n(S)=S
\end{equation}
was validated for all $n\in\N$.
 If $f_i$ is a contractive homeomorphism, 
then  \eqref{eq self-similarity} means that for $n\in\N$, $S$ has a self-similarity; that is, $S$ consists of $|\mathcal{L}_n(\Sigma_k)|=k^n$ parts all topologically similar to $S$. We call \eqref{eq self-similarity} the \emph{ self-similarity criteria}.

In this note, we consider  \IFS where $\Sigma$ stands for an arbitrary sub-shift on a  character set $\mathcal{A}$  such that if not specified $|\mathcal{A}|=k\geq 2$. 

 By substituting $\Sigma$ with $\Sigma_k$ in \eqref{eq H^n(A)}, we define 
 $H^n:\mathcal{C}\rightarrow\mathcal{C}$, which is no longer the iteration of $H$ as mentioned earlier.
In general, \eqref{eq self-similarity} loses its validity.
However, we demonstrate that if $\Sigma$ is a coded sub-shift, that is, a sub-shift generated by a set of words, then
\eqref{eq self-similarity} holds
 for infinitely many $n\in\N$.
 Furthermore, if $\Sigma$ is mixing, then for sufficiently large $n$,
  \eqref{eq self-similarity} is valid. 
  This validity will apply to all 
  $n$, if our 
    coded $\Sigma$
   possesses a specified fixed point. 
  Given the infinite number of subsystems of $\Sigma_k$, that are coded, in cases where 
  $f_u$'s are contractive homeomorphisms,
  there exists an infinite array of self-similarities within the attractor of a an IFS. This offers a basic understanding of why numerous distinct similarities emerge from a set of finite forces $\mathcal{F}$ acting on a space like our universe. 

Assuming \eqref{eq self-similarity} holds for all $n$ in our $\mathfrak{I}=(X,\,\mathcal{F},\,\Sigma)$ with the attractor $S$, then $\mathfrak{I}_S=(S,\,\mathcal{F},\,\Sigma)$ becomes an IFS as well, inheriting many dynamical properties from $\Sigma$. Moreover, for any $\sigma=\sigma_1\sigma_2\cdots\in\Sigma$, define $\mathcal{A}_\sigma$ as the set of characters in $\sigma$ and $\Sigma_\sigma$ as the subsystem of $\Sigma$ obtained from the closure of points in the orbit of $\sigma$.
Then, we consider $\mathfrak{I}_\sigma=(X,\,\mathcal{F}_\sigma,\,\Sigma_\sigma)$ with an attracting set $S_\sigma$ and the non-autonomous system $f_{1,\,\infty}$ with $f_i=f_{\sigma_i}$ and will demonstrate that these two systems share dynamics with $\mathfrak{I}$. Notably, $S=\cup_{\sigma\in\Sigma}S_\sigma$ and when $\sigma$ is transitive, the attracting set of $f_{1,\,\infty}$ is $S$.

\section{Preliminaries}\label{sec preliminaries}
A dynamical system represents an action of a semi-group on a set. However, our focus here is on the "classical" topological dynamical system, where $\N$ or $\Z$ acts on a compact metric space.
Our "non-classical" dynamical systems include iterated function systems and non-autonomous dynamical systems, which will be discussed later. To be precise, consider a compact metric space $X$ with a continuous map $f:X\rightarrow X$. Let $G\in\{\N,\,\Z\}$ and define $G\times X\rightarrow X$ as
$(n,\,x) \mapsto f^{n}(x):=\overbrace{ f \circ f \circ \cdots \circ f}^{n \text{ times}}(x)$.
The \emph{forward orbit} of $x\in X$ is
$\{x,\, f(x),\ldots ,\, f^{n}(x), \ldots \}$ and is denoted by $\mathcal{O}^+(x)$.
If $f$ is a homeomorphism, then $\mathcal{O}_-(x)$, the \emph{backward orbit} of $x$, is defined similarly, and $\mathcal{O}^+_-(x)=\mathcal{O}^+(x)\cup \mathcal{O}_-(x)$ represents the \emph{full orbit} of $x$.
Let $\mathcal{O}(x)$ denote either $\mathcal{O}^+(x)$ or $\mathcal{O}^+_-(x)$ depending on whether $G=\N$ or $G=\Z$ respectively.


\subsection{Symbolic dynamics}\label{symdyn}
We require symbolic dynamics. Therefore, we revisit certain necessary concepts, primarily from \cite{lind2021introduction}.

Consider the set $\mathcal{A}$ to be a non-empty finite set and let $\Sigma_{|\mathcal{A}|}= \mathcal{A}^{\mathbb{Z}}$ (resp. $ \mathcal{A}^{\mathbb{N}}$) represent the collection of all bi-infinite (resp. right-infinite) sequences of symbols from $\mathcal{A}$. The map $\tau$ on $\Sigma_{|\mathcal{A}|}$, known as the \emph{shift map}, is defined by $\tau(\sigma)= \sigma'$ where $\sigma'_{i}=\sigma_{i+1}$. The pair $(\Sigma_{|\mathcal{A}|},\tau)$ is called the \emph{full shift}, and any closed invariant subset $\Sigma$ of $\Sigma_{|\mathcal{A}|}$ is referred to as a \emph{sub-shift} or a \emph{shift space}. A \emph{word} or \emph{block} over $\mathcal{A}$ is a finite sequence of symbols from $\mathcal{A}$.
Denote by
$\mathcal{L}_{n}(\Sigma)$
the set of all admissible words of length $n$ and define
$\mathcal{L}(\Sigma)=\bigcup_{n=0}^{\infty} \mathcal{L}_{n}(\Sigma)$ 
as the \emph{language} of
$\Sigma$.
According to this definition, every word $u$ possesses a right prolongation. This implies that if $u\in\mathcal{L}(\Sigma)$, then there exists some $i\in\mathcal{A}$ such that $ui\in\mathcal{L}(\Sigma)$.
 The left prolongation occurs in sub-shifts of $\mathcal{A}^\Z$ and not necessarily in  $\mathcal{A}^\N$. 
For example, if $\Sigma=\{21^\infty,\,1^\infty\}\subset \{1,\,2\}^\N$, it forms a sub-shift that lacks left prolongation.
To simplify matters, it is often assumed that $\tau$ is totally invariant, meaning $\tau^{-1}(\Sigma)=\Sigma$. This implies that under this assumption, any $u\in\mathcal{L}$ will invariably possess a left prolongation.

For $u \in \mathcal{L}_{k}(\Sigma)$, let the \emph{cylinder} $_{l}[u]_{l+k-1}=[u_{l}\ldots u_{l+k-1}] $ be the set $\{\sigma= \cdots \sigma_{-1} \sigma_{0} \sigma_{1} \cdots \in \Sigma: \sigma_{l} \cdots \sigma_{l+k+1}=u\}$. Let $\bf{F}$ be a collection of blocks over $\mathcal{A}$, which we will think of as being the \emph{forbidden blocks}. For any such $\bf{F}$, define $\Sigma_{\bf{F}}$ to be the set of sequences in $ \mathcal{A}^{\mathbb{Z}} $ which do not contain any block in $\bf{F}$. Now we can say that a sub-shift is a subset $\Sigma$ of a full shift $\mathcal{A}^{\mathbb{Z}} $ such that $ \Sigma = \Sigma_{\bf{F}} $ for some collection $\bf{F} $ of forbidden blocks over $\mathcal{A}$. A shift space $\Sigma$ is \emph{irreducible} if for every ordered pair of words $u,\, v \in \mathcal{L}(\Sigma)$ there is a word $w \in \mathcal{L}(\Sigma)$ so that $uwv \in \mathcal{L}(\Sigma)$. A point $\sigma \in \Sigma$ is \emph{doubly transitive} if every word in $\Sigma$ appears in $\sigma$ infinitely often to the left and the right.

\begin{defn}\label{defn factor code} \cite[\S 1.5]{lind2021introduction}. Let $\Sigma$ (resp. $\Sigma'$) be a shift space over $\mathcal{A}$ (resp. $\mathcal{A}'$) and let $\varPhi : \mathcal{L}_{m+n+1}(\Sigma)\rightarrow \mathcal{A}'$. Then, $\varPhi$ induces a surjective map $\varphi:\Sigma \rightarrow \Sigma'$ ($\varphi = \varPhi^{[-m,\,n]}_{\infty} $) defined by $\sigma'=\varphi(\sigma) $ with $\sigma'_{i}= \varPhi (\sigma_{i-m} \sigma_{i-m+1} \ldots \sigma_{i+n})= \varPhi(\sigma_{[i-m,i+n]})$ which is called the \emph{sliding block code} or simply \emph{shift code} with memory $m$ and anticipation $n$.
\end{defn}
	The map $\varPhi$ is known as the \emph{block map}, and for the case where $m=n=0$, we refer to $\varphi$ as the \emph{1-block factor map}. Note that $\varphi$ factors $(\Sigma,\,\tau)$ over $(\Sigma',\,\tau)$ where $\tau$ represents the shift map. In this scenario, we can say that $\Sigma'$ is a factor of $\Sigma$. A sub-shift $\Sigma$ is considered \emph{conjugate} to $\Sigma'$, denoted by $\Sigma\simeq\Sigma'$, if $\Sigma$ (or $\Sigma'$) factors over $\Sigma'$ (or $\Sigma)$. For any shift code $\varphi:\Sigma\rightarrow\Sigma'$, there exists $\tilde{\Sigma}\simeq \Sigma$ and a $\tilde{\varphi}:\tilde{\Sigma}\rightarrow \Sigma'$ such that $\tilde{\varphi}$ is a 1-block factor map \cite[Proposition 1.5.12]{lind2021introduction}. Therefore, in the qualitative study of shift spaces, it is acceptable to assume that a factor code is a 1-block factor map.

Shift spaces defined by a finite set of forbidden blocks are termed \emph{shifts of finite type} (SFT), and their factors are known as \emph{sofic}.

A word $w \in \mathcal{L}(\Sigma)$ is regarded as \emph{synchronizing} if $uwv \in \mathcal{L}(\Sigma)$ whenever $uw,\,wv \in \mathcal{L}(\Sigma)$. A \emph{synchronized} system is an irreducible shift that possesses a synchronizing word. Every sofic system is synchronized, and all synchronized systems are coded:
\begin{defn}\label{defn coded}
	Let $\Sigma$ represent a shift space over $\mathcal{A}$, and let $\mathcal{W}\subset\mathcal{L}(\Sigma)$. We label the pair $(\Sigma,\,\tau)$ as a \emph{coded shift} if it is the closure of the set of sequences formed by freely concatenating words in $\mathcal{W}$. The set $\mathcal{W}$ is denoted as the \emph{generator} of the coded sub-shift $\Sigma$.
\end{defn}

Like sofics, factors of coded shifts are coded. Additionally, $(\Sigma,\,\tau)$ is coded if and only if $\Sigma$ contains an increasing sequence of irreducible SFT whose union is dense in $\Sigma$.

Coded systems, which are an extension of sofic systems, were first introduced by Blanchard and Hansel in \cite{blanchard1986systemes}. Subsequently, Fiebigs further examined them in \cite{fiebig1992covers}, where they specifically introduced a subclass of coded shifts known as half-synchronized systems that include synchronized systems.

The detailed understanding of half-synchronized systems is not necessary; however, it is important to note that a half-synchronized system is mixing if and only if there exists a generator $\mathcal{W}$ such that $gcd\{w: w \in \mathcal{W}\}=1$ \cite[Theorem 3.4]{dastjerdi2019mixing}.

\section{Iterated functions systems  (IFS) over an arbitrary sub-shift}\label{Sec IFS}
The \emph{conventional iterated function system}, or simply the conventional IFS, 
$(X,\,\mathcal{F})$
comprises a compact metric space
$X$
and a finite set of continuous self-maps
$\mathcal{F}=\{f_{1},\ldots ,\, f_{k}\}$
on
$X.$
We represent this system as
\begin{equation}\label{dtt}
	\mathfrak{I}= (X,\, \mathcal{F},\,\Sigma_{|\mathcal{A}|}),
\end{equation}
where
$\Sigma_{|\mathcal{A}|}$
is the full shift on
$k=|\mathcal{A}|$
symbols.
A ``generic" IFS, or simply an IFS 
refers to the scenario where our sub-shift is no longer a full shift. 
\begin{defn}\label{defn IFS}
	Let $\Sigma$ be a sub-shift on 
	$\mathcal{A}=\{1,\,2,\ldots,k\}$
	and $\mathcal{F}=\{f_{1}, \ldots , f_{k}\}$ a set of $k$ continuous self-maps on the compact metric space $X$.
	Call
	\begin{equation}\label{eq gdtt}
		\mathfrak{I}=(X,\,\mathcal{F},\,\Sigma)
	\end{equation}
	an IFS where for
	$u= u_{1} \cdots u_{n} \in \mathcal{L}(\Sigma)$,
	$f_{u}: X \longrightarrow X $ is given by
	\begin{equation}\label{eq f_u}
		f_u(x)=f_{u_{n}} \circ \cdots \circ f_{u_{1}}(x).
	\end{equation}
	The \emph{orbit of} $x\in X$ is set as
	\begin{equation}\label{eq orbitx}
		\mathcal{O}^{+}(x)=\bigcup_{u\in\mathcal{L}(\Sigma)}f_u(x).
	\end{equation}
	
Consider $\sigma = \sigma_{1} \sigma_{2} \cdots \in \Sigma$, and let the sequence $x, f_{\sigma_{1}}(x), f_{\sigma_{1} \sigma_{2}}(x), \ldots$ be the trajectory of $x$ along $\sigma$. The set of points in this trajectory, denoted as $\mathcal{O}^{+}_\sigma(x)$, is called the forward orbit of $x$ along $\sigma$. This orbit defines a non-autonomous system denoted by
	\begin{equation}\label{eq fsigma}
		f_{\sigma_1,\,\infty}.
	\end{equation}
	
	We say $\mathcal{F}=\{f_{1},\ldots ,\,f_{k}\}$ has property $P$, if all maps in $\mathcal{F}$ have that property.
	If $\mathcal{F}$ is homeomorphism, then
	${\mathcal{O}_\sigma}_-(x)$,
	${\mathcal{O}_\sigma}^{+}_{-}(x)$
	, or
	${\mathcal{O}}_{\sigma}(x)$
	can be defined similarly as in the classical dynamical systems at the beginning of Section \ref{sec preliminaries}. We adopt the notion
	$f_{\sigma_{1},\infty}$
	for all cases.
\end{defn}

Analogous to classical dynamical systems, if $\mathcal{F}=\{f_{1}, \ldots, f_{k}\}$ has property $P$, and if $\mathcal{F}$ is a homeomorphism, then ${\mathcal{O}_\sigma}_-(x)$, ${\mathcal{O}_\sigma}^{+}_{-}(x)$, or ${\mathcal{O}}_{\sigma}(x)$ can be similarly defined. The notation $f_{\sigma_{1}, \infty}$ is adopted for all cases.

Recall that a non-autonomous system is typically denoted by $f_{1,\,\infty}$ \ah{\cite{kolyada1996topological}} and is shown as \eqref{eq fsigma} to emphasize its dependence on $\sigma$. Therefore, the conventional IFS is a classical topological dynamics defined as the action of $\mathcal{L}(\Sigma_{\mid\mathcal{A}\mid})$ on $X$ via $(u,\,x)\mapsto f_u(x)$.
\begin{defn}\label{defn IDS}
	Let $\Sigma$, $\mathcal{A}$ and $\mathcal{F}$ be as in Definition \ref{defn IFS}. Here, we define an \emph{induced dynamical system (IDS)} denoted as $(\Sigma\times X,\,T)$, where
	$T:\Sigma\times X\rightarrow \Sigma\times X$ is
	given by
	\begin{equation}\label{eq T}
		T(\sigma_1\sigma_2\sigma_3\cdots, \, x)=(\tau(\sigma), \,f_{\sigma_1}(x))=(\sigma_2\sigma_3\cdots, \,f_{\sigma_1}(x)).
	\end{equation} 
\end{defn}


An IDS represents a classical dynamical system, where for any $\sigma=\sigma_1\sigma_2\cdots$, the restriction on the second coordinate unveils the dynamics of the non-autonomous system $f_{\sigma_1,\,\infty}$.

Now, we proceed to establish the concept of factoring between two IFSs, beginning with defining factors between IDSs. 
Let $\Sigma$ and $\Sigma'$ denote shift spaces over $\mathcal{A}$ and $\mathcal{A}'$ respectively, with  
$\varphi_1:\Sigma\rightarrow \Sigma'$ as the factor code outlined in Definition \ref{defn factor code}.
Then,  for 
$N\in\N$, there is $\tilde{\Sigma}\simeq \Sigma$, the $N$th higher block shift of $\Sigma$ with an alphabet
$\tilde{\mathcal{A}}=\mathcal{L}_N({\Sigma})$  \cite[\S 1.4]{lind2021introduction}. Let $\tilde{\varphi_1}:\tilde{\Sigma}\rightarrow\Sigma'$ be the related 1-block factor code.
Additionally,  $\psi_1:\tilde{\Sigma}\rightarrow\Sigma$ serves as the conjugacy map,  where
for any $\tilde{a}=u_1u_2\cdots u_N\in \mathcal{L}_N({\Sigma})$, we define
\begin{equation}\label{eq ftilde a}
	f_{\tilde{a}}=f_{u_1},
\end{equation}  $\tilde{\mathcal{F}}=\{f_{\tilde{a}}:\; \tilde{a}\in\tilde{\mathcal{A}}\}$, and $\tilde{\mathfrak{I}}=(X,\,\tilde{\mathcal{F}},\,\tilde{\Sigma})$.
Subsequently, $\psi=(\psi_1,\,id_2):\tilde{\Sigma}\times X\rightarrow \Sigma\times X$  establishes a conjugacy between $(\tilde{\Sigma}\times X,\, \tilde{T})$ and $(\Sigma\times X,\,T)$.
This sets the stage for the forthcoming definition.
\begin{defn}\label{defn factoring}
	Let	$\mathfrak{I}= (X,\mathcal{F},\,\Sigma)$ (resp. 	$\mathfrak{I}'= (X',\mathcal{F}',\,\Sigma')$) 
	where $\mathcal{F}=\{f_{1},\ldots ,\,f_{k}\}$) (resp. 
	$\mathcal{F}'=\{f'_{1},\ldots ,\,f'_{k'}\}$). Then, 
	$(\Sigma'\times X',\,T')$ is a \emph{factor} of $(\Sigma\times X,\,T)$ via a factor map $\varphi=(\varphi_1,\,\varphi_2)$ whenever $\varphi_1:\Sigma\rightarrow \Sigma'$ is a factor code and for any $a'\in\mathcal\mathcal{A}'$, there is $\tilde{a}\in\tilde{\mathcal{A}}$ such that $\varphi_2\circ f_{\tilde{a}}=f'_{a'}\circ \varphi_2$. Call $(\Sigma\times X,\,T)$ an \emph{extension} of 
	$(\Sigma'\times X',\,T')$.

	If $(\Sigma\times X,\,T)$ is also a factor of $(\Sigma'\times X',\,T')$, then these two systems are \emph{conjugate}.
	
\end{defn}
Thus factoring $(\Sigma\times X,\,T)$ over $(\Sigma'\times X',\,T')$, is defined by factoring 
$(\tilde{\Sigma}\times X,\, \tilde{T})\,(\simeq (\Sigma\times X,\,T))$  over $(\Sigma'\times X',\,T')$. Regarding \eqref{eq ftilde a},
this factoring is established whenever $\Sigma'$ is a factor of $\Sigma$ and for any $\sigma=\sigma_1\sigma_2\cdots\in\Sigma$ with $\sigma'=\sigma'_1\sigma'_2\cdots=\varphi_1(\sigma)$, 
\begin{equation}\label{eq commutativity}
	\varphi_2\circ f_{\sigma_i}=f'_{\sigma'_i}\circ\varphi_2.	
\end{equation}
Let $\sigma\in\Sigma$ and $\Sigma_\sigma=\overline{\mathcal{O}(\sigma)}$. Then, $\Sigma_\sigma$ is a sub-shift over
$\mathcal{A}_\sigma\subset \mathcal{A}$, the set of characters appearing along $\sigma$, and define $\mathcal{F}_\sigma$, $T_\sigma$, and 
$(\Sigma_\sigma\times X,\,T_\sigma)$ accordingly.
\begin{defn}
	Let $f_{\sigma_1,\,\infty}$ and $f'_{\sigma'_1,\,\infty}$ be non-autonomous systems defined on $X$ and $X'$ respectively. Then, $f'_{\sigma'_1,\,\infty}$ is a factor of $f_{\sigma_1,\,\infty}$, if  $(\Sigma_{\sigma'}'\times X',\,T'_{\sigma'})$ is a factor of $(\Sigma_\sigma\times X,\,T_\sigma)$ via a factor map $\varphi=(\varphi_1,\,\varphi_2)$ with $\varphi_{1}(\sigma)=\sigma'$.
	
	Let  $\mathfrak{I}$ and $\mathfrak{I}'$ be  as in Definition \ref{defn factoring}. Then, $\mathfrak{I}'$ is a factor of $\mathfrak{I}$ via $\varphi=(\varphi_1,\,\varphi_2)$, if $f'_{\sigma'_1,\,\infty}$ is a factor of $f_{\sigma_1,\,\infty}$ via $\varphi$ for any $\sigma\in\varphi_{1}^{-1}(\sigma')$.
	
	Two non-autonomous systems or two IFS's are conjugate, if any of them is a factor of the other. Extensions are defined as in Definition \ref{defn factoring}.
\end{defn}
By the same reasoning stated before the above definition, 
$f'_{\sigma'_1,\,\infty}$ is a factor of $f_{\sigma_1,\,\infty}$ whenever $\Sigma'_{\sigma'}$ is a factor of $\Sigma_{\sigma}$ and \eqref{eq commutativity} is satisfied.

Factoring non-autonomous systems can be defined when $\sigma$ includes countably many characters with some care. Remember for shifts over a finite alphabet, a sliding block code's existence ensures its image as a shift space. This does not hold for shifts with infinite characters \cite{sobottka2022some}.
\subsection{Mixing properties for an IFS}
For \IFS\!, the definition of the  
\emph{return time} is given by
$$N(U,\,V)=\{n\in\mathbb{N}: U\cap f_u^{-1}V\neq\emptyset,\, u\in\mathcal{L}_n(\Sigma)\}$$ 
  where $U$ and $V$ are
nonempty and open sets in $X$.

\begin{defn}\label{def IFS properties}
	\cite{ahmadi2023iterated}
	Let
	$\mathfrak{I}= (X,\,\{f_{1},\ldots ,\,f_{k}\},\,\Sigma)$ 
	and let $U$ and $V$ be two nonempty open sets in $X$.
	Then
	$\mathfrak{I}$
	is
	\begin{enumerate}
		\item
		\emph{point transitive}, if there is
		$x \in X$, called the transitive point,
		such that
		$\overline{\{f_{u}(x): u \in \mathcal{L}(\Sigma)\}}=X$.
		\item \label{top tra}
		\emph{topological transitive}, if
		$N(U,\,V)\neq\emptyset$.
		\item 
		\emph{totally transitive}, if for any $N\in \N$, 
		$(X,\,\{f_u:\;u\in\mathcal{L}_N(\Sigma)\},\,\Sigma^N)$ is transitive where
		$\Sigma^N$ is the $N$th higher power shift of $\Sigma$ \cite[\S 1.4]{lind2021introduction}.
		\item
		\emph{weak mixing}, 
		if $N(U,\,V)$ is a thick set, that is, 
		$N(U,\,V)$ contains 
		arbitrarily long intervals of $\N$.
		\item
	\emph{mixing}, if $N(U,\,V)$ is cofinite. 
	\end{enumerate}	
\end{defn}

All notions in the above definition are preserved under factoring. Also,
\begin{center}
	mixing $\Rightarrow$ weak mixing $\Rightarrow$ totally transitive  $\Rightarrow$ point transitive.
\end{center}

\begin{rem}\label{rem mixing}
	If we set $X := \Sigma$ and $\mathcal{F} := \{\tau\}$, the classic definition of transitivity, mixing, etc., for a sub-shift emerges. For instance, $(\Sigma, \tau)$ is transitive if, for any $u, v \in \mathcal{L}(\Sigma)$, there exists $w$ such that $uwv \in \mathcal{L}(\Sigma)$; or, it is mixing if there exists $M = M(U, V)$ such that for $n>M$, there is $w \in \mathcal{L}_{n}(\Sigma)$ with $uwv \in \mathcal{L}(\Sigma)$.
\end{rem}

\subsubsection{Periodic points and chaos}
There are various types of periodic points in an IFS (See \cite[Definition 3.1]{dastjerdi2023shift}).
Yet, we adhere to the subsequent definition.
\begin{defn}
	Let \IFS be an IFS.
	A point $x\in X$ is periodic of period $p$ if there exists $\sigma\in\Sigma$ such that $T^p(\sigma, \, x)=(\sigma, \, x)$ where $T$ is the IDS map defined in \ref{eq T}.
\end{defn}

Thus, a point $x\in X$ is periodic of period $p$ if there exists $u=u_1\cdots u_p\in\mathcal{L}(\Sigma)$ with $\sigma=u^\infty\in\Sigma$  and $f_u(x)=x$. Additionally, periodic points map to periodic points under the factor map.

\subsection{ The attractor set of an IFS}
Hereinafter, we assume that the maps in $\mathcal{F}$ are contractions. This means that for $f_i\in\mathcal{F}$, $f_i:(X,\,d) \rightarrow (X,\,d)$, there exists a constant $0\leq r_i < 1$ such that $d(f_i(x),\,f_i(y)) \leq r_i d(x,\,y)$ for all $x,\,y \in X$. Let $\{r_{1},\ldots ,\,r_{k}\}$ be the \emph{ratio list} associated with $\mathfrak{I}= (X,\mathcal{F},\,\Sigma)$ and $\mathcal{C}$ be the set of compact subsets in $X$. Subsequently, $(\mathcal{C}, \,d_{H})$ will be a compact space where $d_{H}$ denotes the Hausdorff metric. We use $B_H(\cdot,\,\epsilon)$ to represent a ball with radius $\epsilon$ in the Hausdorff metric.

In \cite{hutchinson1981fractals},
Hutchinson observed that when
$\Sigma$
represents a full shift, (referred here as the conventional IFS),
$H: \mathcal{C} \rightarrow \mathcal{C}$
is defined by
\begin{equation}\label{eq Hutchinson map}
	H(A)= \bigcup_{i=1}^{|\mathcal{A}|} f_i(A)
\end{equation}
is a contraction. Thus, according to the Contraction Mapping Theorem, there exists a unique
$S \in \mathcal{C}$, known as the attractor,
such that
for any $A\in\mathcal{C}$,
\begin{equation}\label{eq limitgeneral}
	\lim_{n \rightarrow \infty} H^n(A)=S.
\end{equation}
and
\begin{equation}\label{eq H(S)=S}
	H(S)=S.
\end{equation}
The map
$H$
is known as the \emph{Hutchinson operator}, and $H^n(A)$ represents the $n$th iteration of $H$. Hence, if $n\in\N$, then from \eqref{eq H(S)=S}, we obtain $H^n(S)=S$ ($H=H^1$).
We extend the definition of $H^n$ to cases where
$\Sigma$
is no longer a full shift; in that scenario, $H(=H_\Sigma)$, as defined in \eqref{eq Hutchinson map}, is not a map, but rather
\begin{equation}\label{eq F^n general}
	H^n(A)=\bigcup_{u\in\mathcal{L}_n(\Sigma)}f_u(A)\in\mathcal{C}
\end{equation}
with
\begin{equation}\label{eq induction}
	H^{n+1}(A)=\bigcup_{ui\in\mathcal{L}_{n+1}(\Sigma),\,i\in\mathcal{A}}f_{ui}(A)
\end{equation}
is well understood and is non-empty. This is due to the fact that any word in $\mathcal{L}(\Sigma)$ has a right prolongation.
We refer to $H=H^1$ as the "generalized" \emph{Hutchinson operator} (for simplicity, we refer to it as the Hutchinson operator, dropping the subscript $\Sigma$ for $H^n$ and $H^{n+1}$ as mentioned above).
Note that if we define a forbidden word as $f_u(A)=\emptyset$, then we can express $$H^{n+1}(A)=\bigcup_{i=1}^{\mathcal|\mathcal{A}|}\bigcup_{u\in\mathcal{L}_{n}(\Sigma)}f_{ui}(A)=\bigcup_{i=1}^{\mathcal|\mathcal{A}|}f_i(H^n(A)),$$ which is similar to the conventional case definition. We will see that this definition of $H^n(A)$ ensures the limit in \eqref{eq limitgeneral} but not the equality in \eqref{eq H(S)=S}.

\begin{example}\label{ex 12infty}
	(1). Consider $\mathfrak{I}=(X,\,\{f_1,\,f_2\},\,\Sigma)$ where
	$X=[0,\,1]$,
	$\Sigma=\{(12)^\infty,\,(21)^\infty\}$  and $f_1(x)=\frac13 x$, $f_2(x)=\frac13 x +\frac23$.
	Then, $\mathcal{L}_{2n}(\Sigma)=\{(12)^n\}\cup \{(21)^n\}$ and in view of equation $\eqref{eq induction}$, $\mathcal{L}_{2n+1}(\Sigma)=\{(12)^n1\}\cup \{(21)^n2\}$.
	Note that $f_1$ and $f_2$ along with $f_{21}=f_1\circ f_2$ and  $f_{12}=f_2\circ f_1$ are contraction mappings. Hence, by the Contraction Mapping Theorem,
	for $A\in\mathcal{C}$, $\lim_nf_{(21)^n}(A)=\frac14$ and $\lim_nf_{(12)^n}(A)=\frac34$.
	Furthermore, $\lim_n f_{(21)^n2}(A)=f_2(\lim_n f_{(21)^n}(A))=\frac34$ and
	$\lim_n f_{(12)^n1}(A)=f_1(\lim_n f_{(12)^n}(A))=\frac14$.
	Therefore, $S=\lim_n H^n(X)=\{\frac34,\,\frac14\}$;
	however, unlike the situation where $\Sigma$ was a full shift, equation $\eqref{eq H(S)=S}$ does not hold. Here, we have
	$ H(S)=\{\frac14,\,\frac34,\,\frac{1}{12},\,\frac{11}{12}\}$.
	\vskip.3cm
	
	\noindent(2). Note that if $\Sigma$ in the above IFS is the full shift $\{1,\,2\}^\N$, then according to Hutchinson's findings in \cite{hutchinson1981fractals}, both equations $\eqref{eq limitgeneral}$ and $\eqref{eq H(S)=S}$ are met, and it is not hard to see that $S$ corresponds to the standard Cantor set.
	\vskip.3cm
	
	\noindent(3).
	In situations where all maps in
	$\mathcal{F}$ of an IFS
	possess the same fixed point, denoted as
	$x_{0}$,
	then for
	$A \in \mathcal{C}$,
	$S=\lim_{n \rightarrow \infty}H^{n}(A)=\{x_{0}\}$ and equation $\eqref{eq H(S)=S}$ holds for
	any sub-shift
	$\Sigma$.
	An instance is when the maps in
	$\mathcal{F}$
	are commutative.
\end{example}

The following demonstrates that \eqref{eq limitgeneral} holds true.
\begin{prop} \label{prop AX}
	Consider \IFS where $\mathcal{F}$ is a contraction.
	Suppose $A$ and $B$ are compact sets in $X$. Then,  
	$\lim_{n\rightarrow\infty} H^{n}(A)$ exists and
	\begin{equation} \label{eq H^n}
		\lim_{n\rightarrow\infty} H^{n}(A)= \lim_{n\rightarrow\infty} H^{n}(B).
	\end{equation}
	Let $S=\lim_{n\rightarrow\infty}H^n(X)$. Thus, $S$ is unique, and we will have $S\subseteq H^n(S)$ for every $n\in\N$.
\end{prop}
\begin{proof}
	First, we establish the existence of $\lim_{n\rightarrow\infty} H^{n}(X)$.
	For any $v\in \mathcal{L}_{n+1}(\Sigma)$, there exists $a\in\mathcal{A}$ and $u\in\mathcal{L}_n(\Sigma)$ such that $v=au$.
	This implies $f_v(X)=f_u(f_a(X))\subseteq f_u(X)$ and thus $H^{n+1}(X)\subseteq H^n(X)$. Consequently, $$S=\lim_{n\rightarrow\infty}H^{n}(X)=\bigcap_{n\in\N}H^n(X),$$ which is compact as each $H^n(X)$ is compact.
	Furthermore, $S\subseteq H^n(S)$ and as a limit, it must be unique.
	
	To show the equality in \eqref{eq H^n}, it suffices to demonstrate that if $B=X$ and $A$ is an arbitrary compact set, then \eqref{eq H^n} holds.
	Let $S=\lim_{n\rightarrow\infty}H^n(X)$ and consider the sequence $\{u^{(p)}\}_{p\in\N}$ where $u^{(p)}=u_1^{(p)}\cdots u_p^{(p)}\in\mathcal{L}_p(\Sigma)$ with increasing length $p$.
	We have diam$f_{u^{(p)}}(X)<\left(\Pi_{i=1}^p r_{u_i^{(p)}}\right)\text{diam}X$ where $r_{u_i^{(p)}}$ is the contraction ratio of $f_{u_i^{(p)}}$.
	Therefore, for $A\subseteq X$, $\lim_{p\rightarrow\infty}\text{diam}(f_{u^{(p)}}(A))\leq \lim_{p\rightarrow\infty}\text{diam}(f_{u^{(p)}}(X))=0.$
		In particular, if $B(s;\,\epsilon)\subseteq X$ is a neighborhood of a point $s\in S$ containing $f_u(X)$, then it will also contain $f_u(A)\subseteq f_u(X)$. In essence, if any $\epsilon$-neighborhood of $S$ contains $f_u(X)$, it will encompass $f_u(A)$ too, leading to the desired outcome.
\end{proof}

\begin{defn}
	The set $S = \lim_n H^n(X)$ in Proposition \ref{prop AX} is referred to as the \emph{attractor} of the IFS.
\end{defn}

\subsubsection{Abundance of self-similarities in an attractor}
A compact topological space $Y$ is considered \emph{self-similar} if there exists a finite set of non-surjective homeomorphisms $\{h_j: Y \overset{h_j}{\rightarrow} Y; 1 \leq j \leq l\}$ such that $Y = \bigcup_{j=1}^{l} h_j(Y)$. Therefore, when $Y$ is self-similar, parts resemble the whole up to homeomorphism. This leads to the following definition.
\begin{defn} \label{def selfsim}
	Let \IFS be an IFS and let $S$ be its attractor. If for $u \in \mathcal{L}_n(\Sigma)$, $f_u$ is a non-surjective homeomorphism on $S$ and $S = H^n(S) = \bigcup_{u \in \mathcal{L}_n(\Sigma)} f_u(S)$, then we say that $S$ has a self-similarity of order $n$.
\end{defn}

In part 3 of the following example, an IFS is presented such that its associated maps are contractive homeomorphisms on $X$, but there exist some $n \in \N$ where none of $f_u$'s are contractive homeomorphisms on $S$ for $u \in \mathcal{L}_n(\Sigma)$.

\begin{example}
	(1). The Koch curve shown in Figure \ref{fig Koch} serves as the attractor for $$\mathfrak{I}_g = (X, \{g_1, g_2, g_3, g_4\}, \Sigma_4)$$ where $X = [0, 1] \times [0, 1]$.
	The transformations are defined as follows:
	$g_1(x,y) = \frac{1}{3}(x,y)$, $g_2(x,y) = \frac{1}{6}(x - \sqrt{3}y + 2, \sqrt{3}x + y)$,
	$g_3(x,y) = \frac{1}{6}(x + \sqrt{3}y + 3, -\sqrt{3}x + y + \sqrt{3})$, and
	$g_4(x,y) = \frac{1}{3}(x + 2, y)$.
	When $A = [0, 1] \times \{0\} \subset X$, Figure \ref{fig Koch} illustrates $A$, $H(A)$, and $H^2(A)$ from top to bottom.
	For all $n \in \N$, $H^n(S) = S$, as our sub-shift $\Sigma_4$ is a full shift, indicating that $\mathfrak{I}_g$ possesses self-similarities of all orders.
	\begin{figure}
		\centering
		\includegraphics[width=.4\linewidth]{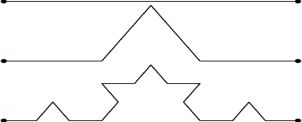}
		\caption{$A$, $H(A)$, and $H^2(A)$ for the Koch curve}\label{fig Koch}
	\end{figure}
	\vskip.3cm
	(2). Later, we will demonstrate the existence of infinitely many compact subsets $S' \subset S$ such that $H(S') = S'$. Let us regard some trivial scenarios. Restriction of $g_1$ and $g_4$ to $[0, 1]$ is $f_1$ and $f_2$ respectively in Example \ref{ex 12infty}. This illustrates that the standard Cantor set $S' \subset S$ through $g_1$ and $g_4$ unveils self-similarities of all orders within $S$. Removing certain functions from $\mathcal{F}$ as done here, reveals only finitely many such self-similarities within $S$. Our primary objective will be to identify self-similarities for compact subsets of $S$ without excluding any map from $\mathcal{F}$, but by exploring the appropriate subsystems of our full shift that fulfill this requirement.
	
	(3). The IFS in Example \ref{ex 12infty}(1) lacks self-similarities of any order for its attractor $S = \{\frac{3}{4}, \frac{1}{4}\}$. There, we demonstrated that it lacks self-similarity of order 1, and the same reasoning shows its lack of order $2n + 1$ as well. For that example, $H^{2n}(S) = f_{(12)^n}(S) \cup f_{(21)^n}(S)$, yet $S = f_{(12)^n}(S) = f_{(21)^n}(S)$, thereby indicating that none of $f_u$, $u \in \mathcal{L}_{2n}(\Sigma)$, is a non-surjective homeomorphism.
	Consequently, by definition, our IFS does not possess self-similarity of order 2n, and by extension, not of any order. This outcome is expected, as a finite set cannot be homeomorphic to its proper subsets.
\end{example}


\begin{rem}\label{rem equivalent for S}
	Let \IFS be an IFS.
	Recall that by Definition \ref{def selfsim}, the attractor $S$ has a self-similarity of order $n$, if $S=H^{n}(S)= \bigcup_{u \in\mathcal{L}_{n}(\Sigma)}f_{u}(S)$ and $f_u$ a contractive homeomorphism on $S$. This implies that for all $k \in \N$, $H^{kn}(S)=S$. For assume $S=H^n(S)$ and for $v\in\mathcal{L}_{2n}(\Sigma)$ write $v=u'u$; $ u',\, u \in \mathcal{L}_{n}(\Sigma) $. Then,
	\begin{align*}
		H^{2n}(S)&= \bigcup_{v\in \mathcal{L}_{2n}(\Sigma)}f_v(S)= \bigcup_{u',\, u \in \mathcal{L}_{n}(\Sigma)}f_{u}(f_{u'}(S))\\
		& = \bigcup_{ u \in\mathcal{L}_{n}(\Sigma)}f_u(\bigcup_{ u' \in \mathcal{L}_{n}(\Sigma)}f_{u'}(S))= \bigcup_{ u \in\mathcal{L}_{n}(\Sigma)}f_u(S)=S.
	\end{align*}
	Now $ H^{kn}(S)=S $ and self-similarity of order $kn$ follows from an induction argument and the fact that $f_v$ is a contractive homeomorphism on $S$ for $v\in\mathcal{L}_{kn}(\Sigma)$.
\end{rem}

Next result shows how self-similarity is preserved by factoring.

\begin{prop}\label{prop H(S) is preserved}
	Let $\mathfrak{I}=(X,\,\mathcal{F},\,\Sigma)$ be an IFS with $\mathcal{F}$ contraction and factoring on $\mathfrak{I}'=(X',\,\mathcal{F}',\,\Sigma')$ via $\varphi=(\varphi_1,\,\varphi_2):\Sigma\times X\rightarrow \Sigma'\times X'$ where $\varphi_1$ is a factor code with memory $m$ and anticipation $n$. Then, $\varphi_2(S)=S'$ where $S$ and $S'$ are the attractors of $\mathfrak{I}$ and $\mathfrak{I}'$ respectively. Moreover, if $H^{\ell+m+n}(S) = S$, then ${H'}^\ell(S')=S'$. Here, $H'$ is the Hutchinson operator for $\mathfrak{I}'$.
\end{prop}
\begin{proof}
	First assume that $\varphi_1$ is a 1-block factor map with block map $\varPhi_1:\mathcal{A}\rightarrow \mathcal{A}'$. Then for $u\in\mathcal{L}_\ell(\Sigma)$, there is $u':= \varPhi_1(u) \in\mathcal{L}'_\ell(\Sigma')$ with $\varphi_2\circ f _{u} = f'_{\varPhi_1(u)}\circ \varphi_2$ and
	\begin{align}
		\varphi_2({H}^{\ell}(X))& = \varphi_2 (\bigcup_{u \in \mathcal{L}_{\ell}(\Sigma)}f_{u}(X)) = \bigcup_{u \in \mathcal{L}_{\ell}(\Sigma)} \varphi_2 f_{u}(X)\nonumber\\
		&= \bigcup_{u \in \mathcal{L}_{\ell}(\Sigma)}f'_{\varPhi_1(u)}(\varphi_2 (X)) = \bigcup_{u' \in \mathcal{L}_{\ell}(\Sigma')} f'_{u'}(X')\label{eq wasati}\\
		& = {H'}^{\ell}(X').\nonumber
	\end{align}
	But $\varphi_2:X\rightarrow X'$ is continuous and so continuous as $\varphi_2:\mathcal{C}\rightarrow\mathcal{C}'$. Thus $\varphi_2(S)=\varphi_2(\lim_\ell H^\ell(X))=\lim_\ell \varphi_2(H^\ell(X))=\lim_\ell {H'}^\ell(X')=S'$. Moreover, assume that $S=H^\ell(S)$. Then by replacing $X$ with $S$ in \eqref{eq wasati} and applying the same routine, we will have $H'^\ell(S')=\varphi_2(H^\ell(S))=\varphi_2(S)=S'$.
	
	If $\varphi_1$ is not a 1-block factor code map, then by Definition \ref{defn factor code}, there are some $(m,\,n)\neq (0,\,0)$ with $\varPhi : \mathcal{L}_{m+n+1}(\Sigma)\rightarrow \mathcal{A}'$. In that case, we write \eqref{eq wasati} as
	\begin{align*}
		\varphi_2({H}^{\ell+m+n}(X))
		&= \bigcup_{u \in \mathcal{L}_{\ell+m+n}(\Sigma)}f'_{\varPhi_1(u)}(\varphi_2 (X)) = \bigcup_{u' \in \mathcal{L}_{\ell}(\Sigma')} f'_{u'}(X')\\
		& = {H'}^{\ell}(X').\nonumber
	\end{align*}
	Now the same proof as above implies that $S'=\varphi_2(S)$ and if $S={H}^{\ell+m+n}(S)$, then $S'={H'}^\ell(S')$.
\end{proof}

When the maps in $\mathcal{F}$ and $\mathcal{F}'$ are all contractions and homeomorphisms, Proposition \ref{prop H(S) is preserved} implies that if $\mathfrak{I}$ has a self-similarity of order $\ell+m+n$, then $\mathfrak{I}'$ has self-similarity of order $\ell$.
In cases where $\mathfrak{I}$ exhibits self-similarity of all orders, the same holds true for $\mathfrak{I}'$.
Typically, we consider the case where $\varphi$ is a 1-block factor code.
With this assumption and Proposition \ref{prop H(S) is preserved}, self-similarity of any order is conserved through factoring.

\begin{cor} \label{cor SFT S}
	Let
	\IFS
	and
	$\Sigma$
	be an irreducible SFT with at least one fixed point.
	Then for all $n\in\N$, $S=H^n(S)$.
	\end{cor}
\begin{proof}
	Let $\tilde{\Sigma}$ be a full shift with $h(\tilde{\Sigma})>h(\Sigma)$. Since $\Sigma$ has a fixed point, by Boyle's Lower Entropy Factor Theorem \cite[Theorem 10.3.1]{lind2021introduction}, there exists a factor code $\varphi_1$ factoring $\tilde{\Sigma}$ with alphabet $\tilde{\mathcal{A}}$ onto $\Sigma$. Without loss of generality, assume $\varphi$ is a 1-block factor code. For $\varPhi_1(\tilde{a})=a$, set $\tilde{f}_{\tilde{a}}:=f_a$ and let $\tilde{\mathcal{F}}=\{\tilde{f}_{\tilde{a}}:\;\tilde{a}\in\tilde{\mathcal{A}}\}$. Now $\varphi=(\varphi_{1},\, id)$ factors $\tilde{\mathfrak{I}}=(X,\,\tilde{\mathcal{F}},\,\tilde{\Sigma})$ onto \IFS\!, and the proof follows from Proposition \ref{prop H(S) is preserved}.
\end{proof}

\begin{example}
	Consider $\mathfrak{I}'=(X,\,\mathcal{F},\,\Sigma')$ where $\mathcal{F}$ is a contractive homeomorphism and $\Sigma'$ is the even shift, generated by $\{21^{2n}:\; n\in\N_0=\N\cup\{0\}\}$. Recall that $\Sigma'$ is an irreducible non-SFT sofic shift and is a factor of the golden shift $\Sigma$ over $\{1,\,2\}$, with a forbidden set of $\{22\}$. Here, the factor code is $\varphi$ with a block map $\varPhi$ defined as $\varPhi(11)=2$, $\varPhi(12)=1$, and $\varPhi(21)=1$ \cite[Example 1.5.6]{lind2021introduction}. Set $\mathfrak{I}=(X,\,\mathcal{F},\,\Sigma)$. Since the golden shift is irreducible and has a fixed point, by Corollary \ref{cor SFT S} and Proposition \ref{prop H(S) is preserved}, both $\mathfrak{I}$ and $\mathfrak{I}'$ have self-similarities of all orders.
\end{example}

The sub-shift in the first IFS in Example \ref{ex 12infty} is an irreducible SFT with a set of forbidden words $\{11,\,22\}$ without a fixed point. There, $S\neq H(S)$, and thus having a fixed point for $\Sigma$ is a necessity for Corollary \ref{cor SFT S}.

Recall that if $\Sigma$ is a coded shift (defined in Definition \ref{defn coded}), then there are increasing SFT's $\Sigma_i$ such that \begin{equation}\label{eq coded union of SFT}
	\Sigma=\overline{\bigcup_{i=1}^\infty\Sigma_i}.
\end{equation}  
This sequence of $\Sigma_i$ is not usually unique. We say $\Sigma$ is \emph{rooted in a fixed point} if there is a sequence of SFT's such that \eqref{eq coded union of SFT} holds and $\bigcup_{i=1}^\infty\Sigma_i$ contains a fixed point.

\begin{example}
	One well-studied class of sub-shifts is density shifts \cite[Definition 3.55]{bruin2022topological}. When a non-trivial density shift with a canonical function has at least two periodic points, then it is a coded shift rooted in a fixed point, namely $0^\infty$. This can be deduced from \cite[Theorem 3.62]{bruin2022topological} by replacing each $\Sigma_i$ with $\Sigma'_i$ so that $\Sigma_i\subseteq \Sigma'_i\subseteq\Sigma$, and such that if $a\in\mathcal{A}$, $uav\in \mathcal{L}({\Sigma_{i}})$, then $ua'v\in \mathcal{L}(\Sigma'_{i})$ for $a'<a$. Note that $\Sigma'_i$ is indeed an SFT, as we are reducing its set of forbidden sets with respect to those of $\Sigma$, and $0^\infty\in\Sigma'_i$ for all $i$.
\end{example}

\begin{lem}\label{lem coded H(S)}
	Assume \IFS with $\mathcal{F}$ contraction and $\Sigma$ a coded shift rooted in a fixed point.
	Then, for all $n\in\N$, we have
	$H^n(S)=S$ .
\end{lem}
\begin{proof}
	Since $\Sigma_i\subseteq\Sigma_{i+1}$, assume that $\Sigma_1$ has a fixed point. By Remark \ref{rem equivalent for S}, it suffices to prove $H(S)=S$. Consider $\mathfrak{I}_i=(X,\,\mathcal{F},\,\Sigma_i)$ and let $\mathcal{A}_i$ be the set of alphabets for $\Sigma_i$. Also,  let $S_i=\lim_n H_i^n(X)$ be the attractor of $\mathfrak{I}_i$, with $H_i$ being the Hutchinson operator for $\mathfrak{I}_i$.
	
	 As $\Sigma_i$ is increasing, we have $S_i\subseteq S_{i+1}$, implying $\{S_i\}_i$ has a limit $S^*$ in the compact space $(\mathcal{C},\,d_H)$. We show that $S=S^*$. Note that $\mathcal{L}_n(\Sigma)=\bigcup_{i=1}^\infty \mathcal{L}_n(\Sigma_i)$, and as $\mathcal{L}_n(\Sigma_i)\subseteq \mathcal{L}_n(\Sigma_{i+1})$ and $\mathcal{L}_n(\Sigma)$ is finite for all $n$, there exists $n_i$ such that $\mathcal{L}_n(\Sigma)=\mathcal{L}_n(\Sigma_{n_i})$ for $n\geq n_i$. By selecting a suitable subsequence of $\{\Sigma_i\}_i$, we can assume $\mathcal{L}_n(\Sigma)=\mathcal{L}_n(\Sigma_{n})$. In particular, $\mathcal{A}=\mathcal{A}_i$ for all $i$.
	
	For any $\epsilon$-neighborhood of $S$, denoted $B_H(S;\,\epsilon)$, there exists $N$ such that for $n>N$, $H^n(X)\in B_H(S;\,\epsilon)$. This implies $S_i\subseteq H_i^n(X)\in B_H(S;\,\epsilon)$ for such $n$, leading to $\lim_iS_i=S$ and thus $S=S^*$ as desired.
	
	To demonstrate $S=H(S)$, from Corollary \ref{cor SFT S}, for all $i$, we have $S_i=H(S_i)=\bigcup_{j=1}^{|\mathcal{A}|}f_j(S_i)$. Therefore,
	\begin{align*}
		S=\lim_i S_i&=\lim_i(\bigcup_{j=1}^{|\mathcal{A}|}f_j(S_i))\\
		&=\bigcup_{j=1}^{|\mathcal{A}|}f_j(\lim_i S_i)=\bigcup_{j=1}^{|\mathcal{A}|}f_j(S)=H(S).
	\end{align*}
\end{proof}
\begin{rem}
	A coded shift may possess a fixed point without being rooted in a fixed point. For example, let $\Sigma$ be generated by $\mathcal{W}=\{21^{2n-1}:\;n\in\N\}$. This $\Sigma$ is not rooted in a fixed point, despite $1^\infty\in \Sigma$. Observe that $\Sigma$ is a non-mixing synchronized system with a
	synchronizing word $2$. It is non-mixing for if $u=v=21$, then any $w$ with $uwv\in\mathcal{L}(\Sigma)$ has even length violating definition of mixing in         Remark \ref{rem mixing}.
	On the other hand, a synchronized coded shift is mixing if and only if it contains at least one generator $G$ such that gcd$\{|w|:\; w\in G\}=1$ \cite[Theorem 4.1]{dastjerdi2019mixing}. Similar to the proof of Lemma \ref{lem coded H(S)}, one can take $\Sigma=\overline{\bigcup_{i=1}^\infty\Sigma_i}$ with
	$\Sigma_i$ an irreducible SFT containing a fixed point. Then by Boyle's Lower Entropy Theorem,
	$\Sigma_1$ is a factor of a full shift and thus mixing and must have a generator $\mathcal{W}_1$ with gcd$\{|w|:\;w\in \mathcal{W}_1\}=1$. Extend $\mathcal{W}_1$ to a generator $\mathcal{W}$ for $\Sigma$ with   gcd$\{|w|:\;w\in \mathcal{W}\}=1$ to see that $\Sigma$ will be mixing which is absurd.
\end{rem}

When
$\Sigma$
is the full shift, then for
$n \in \mathbb{N}$,
$S= H^n(S)=\bigcup_{u \in \mathcal{L}_{n}(\Sigma_{|\mathcal{A}|})}f_{u}(S)$.
However, $|\mathcal{L}_{n}(\Sigma_{|\mathcal{A}|})|=|\mathcal{A}|^n$.
So when $\mathcal{F}$ is a homeomorphism, $S$ is a union of $|\mathcal{A}|^n$ parts, each similar to itself.
For a general coded
$\Sigma$,
we have the following instead.
\begin{prop}\label{prop coded mixing}
	Let
	$\Sigma$
	be a coded shift in \IFS\!.
	Then for infinitely many $N\in\N$
	\begin{equation}\label{eq S=fu(S)}
		S= \bigcup_{u \in \mathcal{L}_{N}(\Sigma)}f_{u}(S).
	\end{equation}
	Moreover, if $\Sigma$ is mixing and half-synchronized, then there is $M\in\N$ such that \eqref{eq S=fu(S)} holds for $N\geq M$.

	Hence, when
	$\mathcal{F}$
	is a homeomorphism, then for infinitely many
	$N$,
	$\{f_{u}: u \in \mathcal{L}_{N}(\Sigma)\}$
	reveals some self-similarities for
	$S$.
\end{prop}
\begin{proof}
	Let $w$ be a member of an arbitrary generator $\mathcal{W}$ of $\Sigma$ with $|w|=N$ and
	consider
	$\mathfrak{I}_{w}=(X,\,\mathcal{F}_{w},\,\Sigma^N )$ where
	$\mathcal{F}_{w}=\{f_{u}: u \in \mathcal{L}_{N}(\Sigma)\}$ and let $S_w$ and $H_w$ be the attractor and the Hutchinson operator for $\mathfrak{I}_w$ respectively.
	Note that $w$ is a character of $\Sigma^{N}$ and
	$(\Sigma^{N},\, \tau^{N})$ is a coded shift with
	$\tau^{N}(w^{\infty})= w^{\infty}$; thus
	$(\Sigma^{N}, \tau^{N})$
	is rooted in
	$w^{\infty}$ and
	by applying Lemma \ref{lem coded H(S)}, \eqref{eq S=fu(S)} holds.
	Since $\mathcal{L}_n(\Sigma^N)=\mathcal{L}_{nN}(\Sigma)$, $H_w^n(X)= H^{nN}(X)$ and so
	$S_w=\lim_n H_w^n(X)= \lim_n H^{nN}(X)=S$. The last inequality arises from the fact that $\{H^{nN}(X)\}_n$ is a subsequence of the convergent sequence $\{H^n(X)\}_n$.
	But if
	$w \in \mathcal{W}$,
	then
	$w'= \mathcal{W} \bigcup \{w^{\ell}: \ell \in \mathbb{N}\}  $
	is also a generator for
	$\Sigma$
	and hence
	\eqref{eq S=fu(S)}
	holds for infinitely many
	$N$. (This last part can be deduced from Remark \ref{rem equivalent for S} as well.)
	
	Now assume that
	$\Sigma$
	is a half-synchronized mixing shift. Then, there is a generator
	$\mathcal{W}$
	with
	$w_{1},\, w_{2} \in \mathcal{W}$
	and
	$gcd(w_1,\, w_2)=1$ \cite[\S 3]{dastjerdi2019mixing}.
	By this fact and  by an application of B\'ezout Theorem,
	for any sufficiently large
	$M$,
	there are
	$a_M , b_M \in \mathbb{N}$
	such that
	$M=a_M\mid w_1 \mid + b_M \mid w_2\mid$. This means
	$\mathcal{W} \bigcup \{w_{1}^{a_M} w_{2}^{b_{M}}\}  $
	is again a generator containing a word of length $M$ and the
	proof follows by the above argument.
\end{proof}
Later in Example \ref{ex rooted required}, we will see that there is an IFS with a mixing coded sub-shift that does not have self-similarity of some orders. Thus to have
such self-similarities,
the property of being rooted in the fixed point as in Lemma \ref{lem coded H(S)} is necessary and mixing alone does not suffice.

\subsubsection{Equivalent characterization of the attractor}
The initial characterization is that $S$ can be interpreted as the limit points of any $x\in X$ of some specific sequence of $f_u(x)$'s. For reference purposes, we state this as a remark.

\begin{rem}\label{rem u<v}
	For $n\in\N$, consider ``<" to be the lexicographic order in $\mathcal{L}_n(\Sigma)$ (though any order suffices, we opt for specificity). Extend this to words in $\mathcal{L}(\Sigma)$ by defining
	\begin{equation*}
		u<v \text{ whenever }
		\begin{cases}
			u<v, & u,\, v\in\mathcal{L}_n(\Sigma),\\
			|u|<|v|, & u,\, v\not\in\mathcal{L}_n(\Sigma),
		\end{cases}
	\end{equation*}
	and denote $\mathcal{L}({\Sigma})=\{u^{(1)},\,u^{(2)},\,\ldots\}$ where $u^{(i)}<u^{(i+1)}$.
	Now for $x\in X$, $\{x\}$ is compact and
	$\lim_m H^m(\{x\})=S$.
	Thus, for $\epsilon>0$, $s\in S$, and sufficiently large $m$, $B(s;\,\epsilon)\cap H^m(\{x\})\neq\emptyset$, implying there exists
	at least one lengthy $u^{(n)}$ such that $f_{u^{(n)}}(x)\in B(s;\,\epsilon)$. This signifies
	that in the space $(X,\,d)$, $S$ serves as the limit set of $\{f_{u^{(n)}}(x)\}_{n\in\N}$.
	One can substitute $\{x\}$ with any $A\in\mathcal{C}$ to reach the same deduction. Namely, $s\in S$ if and only if for $\epsilon>0$, there exists $u$ with sufficiently large length such that $f_u(A)\subset B(s;\,\epsilon)\subset X$ or
	$f_u(A)\in B_H(\{s\};\,\epsilon)\subset \mathcal{C}$ for $A\in \mathcal{C}$.
	Here $ B_H(\{s\};\,\epsilon)$ stands for a ball in $\mathcal{C}$ centered at $\{s\}$.
\end{rem}

\begin{defn}\label{defn Ssigma}
	Let
	$\sigma=\sigma_1\sigma_2\cdots\in\Sigma$ and define
	$$\Sigma_{\sigma}
	=\overline{\{\tau^{n}(\sigma):n \in \N\}}.$$
	Then,
	$\Sigma_{\sigma}$
	forms a sub-shift and
	$\mathfrak{I}_{\sigma} = (X,\, \mathcal{F}_\sigma,\,\Sigma_{\sigma}) $ represents an IFS where  $\mathcal{F}_\sigma\subset\mathcal{F}$ includes  maps like $f_\ell$, $\ell=\sigma_i$ for some $i$.
	Consider
	$S_{\sigma}$
	as the corresponding attractor for $\mathfrak{I}_{\sigma}$.
	Likewise, for any non-empty $A\in\mathcal{C}$ within $(\mathcal{C},\,d_H)$, define
	\begin{align}\label{eq limf}
		f_\sigma(A)&=\text{ limit set of }\{f_{\sigma_1\cdots\sigma_m}(A):\; m\in\N\}\nonumber\\
		&=\{f_{\sigma_1\cdots\sigma_m}(A):\; m\in\N\}'.
	\end{align}
\end{defn}
Since $f_{\sigma_1\cdots\sigma_m}(A)\in\mathcal{C}$, $f_\sigma(A)$ is non-empty in $\mathcal{C}$.
\begin{prop} \label{prop gifs}
	Let	\IFS with $\mathcal{F}$  contraction and $S$ its attractor.
		Then,
	\begin{enumerate}
		\item\label{Sigma'}
		if $\bar{\Sigma}$ is a subsystem of $\Sigma$, then $\bar{S}\subseteq S$ where $\bar{S}$ is the attractor of $(X,\,\mathcal{F},\,\bar{\Sigma})$.
		\item \label{item set}
		for any $\sigma\in\Sigma$ and any $A\in\mathcal{C}$, $S_\sigma=f_\sigma(A)$.
		In particular,
		 for a transitive point $t=t_1t_2\cdots\in\Sigma$, $S=S_t=f_t$.
		\item \label{prop cup Sigma}
		for any  $A\in\mathcal{C}$
		\begin{equation}\label{eq Ssigma}
			S=\bigcup_{\sigma\in\Sigma}f_\sigma(A).
		\end{equation}
		\item \label{item fixed & periodic}
$f_\sigma(A)$ is finite whenever there are	 $v,\,u=u_1\cdots u_{m}\in\mathcal{L}(\Sigma)$
		such that $\sigma=vu^\infty$. In that case, 
		$S_\sigma =\{x_0, f_{u_{1}}(x_0), \ldots ,f_{u_{1} \ldots u_{m-1}}(x_0)\}$
		where
		$x_0$
		is the unique fixed point of $f_u$.
		\item \label{item dense peiodic}
		if periodic points are dense in $\Sigma$, then periodic points are dense in $S$ as well.
			\end{enumerate}
\end{prop}

\begin{proof}
	\begin{enumerate}	
		\item
		Since 
		$\mathcal{L}_{n}(\bar{\Sigma})\subseteq \mathcal{L}_{n}(\Sigma)$,
		the result follows from definition of $H^n$ in \eqref{eq F^n general}.
		\item
		Let $\epsilon>0$ and $N$ so large that for $n>N$, $H_\sigma^n(X)\in B_H(S_\sigma;\,\epsilon)$ where $H_\sigma$ is the Hutchinson operator for  $\mathfrak{I}_\sigma$.
		 Then for such $n$, $f_{\sigma_1\cdots \sigma_n}(X)\subseteq H_\sigma^n(X)\in\mathcal{C}$.
		 Thus $\{f_{\sigma_1\cdots \sigma_n}(A)\}_n'\subseteq\{f_{\sigma_1\cdots \sigma_n}(X)\}_n'\subseteq \{H_\sigma^ n(X)\}'=S_\sigma$
		  where for  $B$, $B'$ is the set of limit points in $(\mathcal{C},\,d_H)$.
		 				On the other hand, let
		$s \in S_{\sigma}\subseteq X$ and 	$\epsilon >0$.
		Use above remark and 
		pick
		$u \in \mathcal{L}_{n}(\Sigma_{\sigma})$
		to be any word so that 
		$f_{u}(X) \subseteq B(s;\,\epsilon)$.
		Then
		$u$
		is a subword of
		$\sigma$
		say
		$u=\sigma_{\ell} \cdots \sigma_{{\ell}+\mid u\mid}$
		for some
		$\ell$.
		Now for $A\in\mathcal{C}$, 
		$f_{\sigma_{1} \cdots \sigma_{\ell-1}u} (A)\subseteq
		f_{\sigma_{1} \cdots \sigma_{\ell-1}u} (X) \subseteq f_{u}(X) \subseteq B(s;\,\epsilon)$
		and so
		$s \in f_\sigma(A)$.
		
		\item 
		By part \ref{Sigma'} above,
		$\bigcup_{\sigma \in \Sigma}S_{\sigma} \subseteq S$.
				For the proof, we show that if $s \in S$, then there is $\sigma \in \Sigma$ such that $s \in S_{\sigma}$.
		By above remark, for $n \in \N$, there is a sequence of words $u^{(n)}$ of strictly increasing length 
		such that
		$f_{u^{(n)}}(X) \subseteq B(s,\frac{1}{n})$ and as a result
		$\lim_{n\rightarrow\infty} f_{u^{(n)}}(X)=\{s\}$.
		Pick
		$\sigma^{(n)}\in\Sigma$
		whose initial subword is
		$u^{(n)}$
		and
		choose a subsequence
		$\{\sigma^{(n_{j})}\}_{j \in \N}$
		of
		$\{\sigma^{(n)}\}_{n \in \N}$
		so that
		$\sigma^{(n_i)} \rightarrow \sigma=\sigma_1\sigma_2\cdots$.
		This is possible in a compact set $\Sigma$.
		Then, $f_{\sigma_1\cdots\sigma_m}(X)$ will meet any neighborhood of $s$ for infinitely many $m$ and so
		$s \in S_{\sigma}$.
		\item\label{item Sigma finite}
Let $\sigma=vu^\infty$.		Using part \ref{item set}, it suffices to show that $S_\sigma$ is finite. We have 
	 $$\Sigma_\sigma=\overline{\{\tau^n(vu^\infty):\; n\in\N\}}=
	 \{\tau^i(vu^\infty):\; 0\leq i\leq|v|\}\cup\{u^\infty\}$$ and so $|\Sigma_\sigma|\leq |v|+1$.
		 In particular, $\mathcal{L}_n(\Sigma_\sigma)$ as well as $H^n(X)$  are uniformly finite and so $\lim_nH^n(X)=S_\sigma$ is finite.  
		
	For the second part, let $x_{0} \in X$ be the unique fixed point of $f_u$.
	If $\sigma=u^{\infty}$ is a periodic point, then $\Sigma_\sigma$ consists of just a cycle and the proof follows by definition. 
	In fact then $x_0\in\Sigma_\sigma$ and for any $\ell\in\N$,
	$f_{u^{\ell}u_{1} \cdots u_{i}} (x_{0})= f_{u_{1} \cdots u_{i}}(f_{u^{\ell}} (x_{0}))\in \Sigma_\sigma$,  $1\leqslant i < p$. 
	When $\sigma= vu^{\infty}$, then $f_{v}(A)$ is again compact and $f_{\sigma}(A)= f_{u^{\infty}}(f_{v}(A))=S_\sigma$.
	\item\label{item dense}
	Let $s\in S$ and by applying \eqref{eq Ssigma}, choose $\sigma$ such that $s\in f_\sigma(X)$.
	Let $\{\sigma^{(m)}\}_m$ be a sequence of periodic points in $\Sigma$ with $\lim_m \sigma^{(m)}\rightarrow \sigma$.
	By part \ref{item Sigma finite}, $S_{\sigma^{(m)}}$ is a cycle in $S$ and so $S_{\sigma^{(m)}}\subseteq S$. Now the conclusion follows from the fact that $\lim_m S_{\sigma^{(m)}}=S_\sigma$ in $(\mathcal{C},\,d_H)$.
	\end{enumerate}
\end{proof}

In the case of \ref{item fixed & periodic} in the proposition above, it is important to note that having $S_\sigma$ finite does not always mean that $\Sigma$ is also finite. This distinction becomes clear when considering a situation where $X$ is a singleton and $\Sigma$ is a non-trivial sub-shift. Additionally, it is common to encounter multiple maps with the same fixed point, leading to the existence of several IFS with identical attractors, as demonstrated below.

\begin{example}\label{ex identical attractor}
	Here is a specific example that serves as a prototype for various similar cases. Consider $\mathfrak{I}_{\mathcal{F}}=(X,\,\mathcal{F},\,\Sigma)$ as an IFS, with $X=[0,\,1]$, $\Sigma$ being an SFT generated by $\mathcal{W}=\{w_1=12, w_2=123\}$, and $\mathcal{F}=\{f_1,\,f_2,\,f_3\}$ where
		$$f_1(x)=\frac14x+\frac18,\quad
	f_2(x)=\frac13x+\frac13,\quad
	f_3(x)=\frac13x+\frac14.$$
	Similarly, let $\mathfrak{I}_{\mathcal{G}}=(X,\,\mathcal{G},\,\Sigma)$ represent another IFS defined on the same space $X$ with the same sub-shift $\Sigma$, where $\mathcal{G}=\{g_1,\,g_2,\,g_3\}$ is defined such that $g_{w_1}=f_{w_1}=\frac{1}{12}x+\frac38$ and $g_{w_2}=f_{w_2}=\frac{1}{36}x+\frac38$.
	It is worth noting that there exist uncountably many such $\mathcal{G}$ where $\mathcal{G}$ is a homeomorphism. For instance, for $\frac{2}{17}<a<\frac12$, the following mappings satisfy the given conditions:
	$$g_1(x)=ax+a ,\quad g_2(x)= \frac{1}{12a}x+\frac{7}{24}\quad,\quad g_3(x)=f_3(x).$$
	In this setup, our sub-shift forms an SFT and includes a set of dense periodic points $u^\infty$, where $u$ represents a finite combination of words in $\mathcal{W}$.
	For all such $u$, it holds that $f_u=g_u$ (since $f_{w_i}=g_{w_i}$), and due to the fact that   $\mathcal{F}$ and $\mathcal{G}$ are homeomorphisms, the only periodic points within their respective attractors lie in $S_{u^\infty}$.
	This inference indicates that both attractors share an identical set of periodic points. Therefore, based on part \ref{item dense} of Proposition \ref{prop gifs} and the compactness of attractors, $\mathfrak{I}_\mathcal{F}$ and $\mathfrak{I}_\mathcal{G}$ possess identical attractors.
\end{example}

The subsequent example illustrates instances where IFS with a coded mixing sub-shift and contractive homeomorphism $\mathcal{F}$ exist without exhibiting self-similarity of all orders (refer to Proposition \ref{prop coded mixing}).

\begin{example}\label{ex rooted required}
	Consider $\mathfrak{I}_\mathcal{F}$ with the attractor $S$ as our IFS from Example \ref{ex identical attractor}.
	Given that gcd$\{|w_i|:\, w_i\in\mathcal{W}\}=1$,  $\Sigma$ is mixing. Besides, for any $u\in\mathcal{L}(\Sigma)$, let $x_u\in[0,\,1]$ represent the unique fixed point of $f_u$.
	A straightforward computation yields
	\begin{equation}\label{eq chain of fixed}
		0<x_1<x_{231}<x_{21}<x_3<x_{123}<x_{312}<x_{12}=x_{w_1}<x_2<1.
	\end{equation}
	As previously mentioned in Example \ref{ex identical attractor},
	$S=\overline{\bigcup_u S_{x_u}}$ where $u$ consists of any finite concatenation of $w_1$ and $w_2$.
		However, if any such $u$ contains $w_2$ as a sub-word, then $x_u<x_{w_1}$. This is due to the fact that $f_u$ is an orientation-preserving homeomorphism, and every point in $X$ is attracted to $x_u$ under $f_u$. Therefore, $x_{w_1}=\max S$ and on the other hand $f_2(x_{w_1})>x_{w_1}$ (refer to \eqref{eq chain of fixed}).
		Consequently, $f_2(x_{w_1})\not\in S$; or equivalently, $H(S)\neq S$, implying that $\mathfrak{I}_{\mathcal{F}}$ lacks self-similarities of all orders.
	\end{example}
	
	\subsubsection{Totally disconnected attractors}
	A finite $\Sigma$ is totally disconnected. On the other hand, when infinite and irreducible, it forms a Cantor set, also again a totally disconnected set.
	In certain scenarios, this attribute may extend to the attractor. The case of a finite $\Sigma$ was addressed in part \ref{item Sigma finite} of Proposition \ref{prop gifs}. Here, we provide some sufficient conditions for the attractor to be totally disconnected.
	An evident instance is when $X$ is totally disconnected. In addition, Hutchinson in \cite[Section 3.1]{hutchinson1981fractals} demonstrated that for $(X,\,\mathcal{F},\,\Sigma_{|\mathcal{A}|})$ with $r_i$ as the contraction ratio of $f_i\in\mathcal{F}$, if
	\begin{equation} \label{eq hut}
		\sum_{i=1}^{|\mathcal{A}|} r_{i} < 1,
	\end{equation}
	then $S$ is totally disconnected.
	Part \ref{Sigma'} of Proposition \ref{prop gifs} suggests that having \eqref{eq hut} suffices for $S$ to be totally disconnected for any $\Sigma$.
	We now introduce another condition.
	
	\begin{prop} \label{prop homeomorphism}
		Let \IFS  be an IFS, $\mathcal{F}$ a contraction and homeomorphism, and $\Sigma$ a sub-shift. Moreover, assume there exists $n\in\N$ such that for any distinct $u$, $v$ in $\mathcal{L}_n(\Sigma)$,
		\begin{equation}\label{eq f_u f_v}
			f_u(X)\cap f_v(X)=\emptyset.
		\end{equation}
		Then for $m>n$ and $u$, $v$ in $\mathcal{L}_m(\Sigma)$, \eqref{eq f_u f_v} holds true, and $S$ is totally disconnected. Furthermore, if $\Sigma$ is totally invariant ($\tau^{-1}(\Sigma)=\Sigma$), then for any $i\neq j$,
		\begin{equation}\label{eq SEmpty}
			f_i(S)\cap f_j(S)=\emptyset.
		\end{equation}
	\end{prop}
\begin{proof}
	Assume \eqref{eq f_u f_v} and
let  $u=u_1u_2\cdots u_{n+1}$ and $v=v_1v_2\cdots v_{n+1}$ be two different words in $\mathcal{L}_{n+1}(\Sigma)$. Then either  $u_2\cdots u_{n+1}\neq v_2\cdots v_{n+1}$ or $u_2\cdots u_{n+1}= v_2\cdots v_{n+1}$. 
For the former,  
 $u_{2}\cdots u_{n+1}$ and $v_{2}\cdots v_{n+1}$  are two different words in $\mathcal{L}_n(\Sigma)$ and so 
\begin{align*}\label{eq reversing}
	f_u(X)\cap f_v(X)=&f_{u_{2}\cdots u_{n+1}}(f_{u_1}(X))\cap f_{v_{2}\cdots v_{n+1}}(f_{v_1}(X))\\
	\subseteq&\,f_{u_{2}\cdots u_{n+1}}(X)\cap f_{v_{2}\cdots v_{n+1}}(X)=\emptyset
\end{align*}
where the last equality follows by hypothesis. For
 the latter,  $u\neq v$ implies that $u_1\neq v_1$ and so $f_{u_1\cdots u_{n}}(X)\cap f_{v_1\cdots v_{n}}(X)=\emptyset$. But $f_{u_{n+1}}$ is a homeomorphism and so
 $$f_u(X)\cap f_v(X)=f_{u_{n+1}}(f_{u_1\cdots u_{n}}(X))\cap f_{u_{n+1}}(f_{v_1\cdots v_{n}}(X))=\emptyset.$$
Thus by induction, for $m>n$, \eqref{eq f_u f_v} holds.
To see that $S$ is totally disconnected, let $n$ be as above and let  $s_1$ and $s_2$ be two  points in $S$ with $d(s_1,\,s_2)=\epsilon>0$. We have $S=\lim_m H^m(X)$ and $S\subseteq H^m(X)$. Choose $m>n$ so that $d_H(H^m(X),\,S)<\frac{\epsilon}{2}$
and choose $u^{(i)}$ in $\mathcal{L}_m(\Sigma)$ such that $s_i\in f_{u^{(i)}}(X)$ for $i\in\{1,\,2\}$. By this choice of $m$, $u^{(1)}\neq u^{(2)}$ and by above result $f_{u^{(1)}}(X)\cap f_{u^{(j)}}(X)=\emptyset$, $j\neq 1$ where  $u^{(j)}\in\mathcal{L}_m(\Sigma)$. Since $A:=f_{u^{(1)}}$ and  $B:=\cup_{j\neq 1}f_{u^{(j)}}(X)$ are disjoint compact sets in $X$, there are open sets $V_A$ and $V_B$ such that $s_1\in A\subset V_A$, $s_2\in B\subset V_B$ with $V_A\cap V_B=\emptyset$ and $S\subseteq H^m(X)\subset V_A\cup V_B$. Thus $S$ is totally disconnected.

Now let 
$\Sigma$
be totally invariant and
$i\neq j$.
Then for any
$m \in \mathbb{N}$,
there are words in
$\mathcal{L}_{m}(\Sigma)$
terminating at
$i$ and $j$
respectively.
Choose $n$ so large to have \eqref{eq f_u f_v} satisfied for $u$ and $v$ in $\mathcal{L}_n(\Sigma)$.
 Then,
\begin{equation}
f_i(S)\cap f_j(S)\subseteq\\
\left(\bigcup_{ui\in\mathcal{L}_{n+1}(\Sigma)} f_{ui}(X)\right)\bigcap \left(\bigcup_{vj\in\mathcal{L}_{n+1}(\Sigma)} f_{vj}(X)\right)=\emptyset.
\end{equation}
 \end{proof}
Note that having $\Sigma$ totally invariant is necessary in the above proposition. For instance, if $\Sigma=\{21^\infty,\,1^\infty\}$, then for $m>1$, there are not words in $\mathcal{L}_m(\Sigma)$ terminating at 2. 

The following shows that  \eqref{eq hut} and \eqref{eq f_u f_v} are independent criteria for having a totally disconnected attractor.

\begin{example}\label{eg 2egs}
	Consider $(X,\,\mathcal{F},\,\Sigma)$ where $|\mathcal{A}|\geq 2$ and all maps have a common fixed point $s_0$. Then, regardless of the chosen $r_i$'s or the subset $\Sigma\subseteq \Sigma_{|\mathcal{A}|}$, one has $S=\{s_0\}$. Such an IFS does not satisfy \eqref{eq f_u f_v}; however, we may choose $r_i$ that satisfy \eqref{eq hut}.
	
	On the other hand, let $\Sigma$ be the full shift and let $X= I \times I$ where $I$ is the unit interval. Also, let $\mathcal{F}= \{f_{1},\, f_{2},\, f_{3}\}$ be defined as $f_i(x,\,y)=r(x-a_i,\,y-b_i)+(a_i,\,b_i)$, $\frac13\leq r<\frac12$ with fixed points $(a_i,\,b_i)$ where $(a_1,\,b_1)=(0,\,0)$, $(a_2,\,b_2)=(1,\,0)$, and $(a_3,\,b_3)=(0,\,1)$. Then, this IFS satisfies \eqref{eq f_u f_v} but not \eqref{eq hut}.
\end{example}

\subsection{Attractor as the phase space of an IFS}
We observed in Example \ref{ex 12infty} that in an IFS, the attracting set may not be preserved by maps in $\mathcal{F}$. However, if $\mathcal{F}$ preserves $S$ such as in Lemma \ref{lem coded H(S)} where $\Sigma$ is a coded shift rooted in a fixed point, then ${\mathfrak{I}_{S}=} (S,\, \mathcal{F} , \,\Sigma)$ is an IFS again. Denote IDS$_S$ and $T_S$, the associated IDS and IDS map for $\mathfrak{I}_S$ respectively.
\begin{prop} \label{prop irifs}
	Let $\mathfrak{I}=(X,\, \mathcal{F},\,\Sigma)$ and $\mathcal{F}$ contraction. Assume that $\mathcal{F}$ preserves the attractor set $S$. Then,
	\begin{enumerate}
		\item \label{tra prop}
		any point in $S$ is transitive in $\mathfrak{I}_S=(S,\,\mathcal{F},\,\Sigma)$.
		\item \label{item all s are trans}
		if $\Sigma$ is irreducible and $t\in\Sigma$ is transitive, then $(t,\, s)\in \Sigma\times S$ is a transitive point for $T_S$. In particular, if $\Sigma$ is a minimal sub-shift, then IDS$_S$ is a minimal system.
	\end{enumerate}
\end{prop}
\begin{proof}
	\begin{enumerate}
		\item
		Let $s \in S$ be arbitrary. We must show that
		\begin{equation}\label{eq minimal S}
			S=\overline{\{f_u(s):\;  u\in\mathcal{L}(\Sigma)\}}.
		\end{equation}
		By Remark \ref{rem u<v}, $S$ is the limit set of $\{f_{u^{(n)}}(s) \}_{n \in \mathbb{N}}$ and by the hypothesis, $f_{u^{(n)}}(s) \in S$, so \eqref{eq minimal S} holds.
		\item
		This is a trivial consequence of part \ref{tra prop}.
	\end{enumerate}
\end{proof}

The following example shows that irreducibility of $\Sigma$ is necessary but not sufficient for part \ref{item all s are trans} of Proposition \ref{prop irifs}.
\begin{example}
	Let \IFS where $\Sigma=\{1^\infty,\,2^\infty\}$ and for $i=1,\,2$, $f_i$ a contraction with $f_i(x_i)=x_i$, $x_1\neq x_2$. Then $\mathcal{F}$ preserves $S=\{x_1,\,x_2\}$ and conclusions of parts 2 and 3 in Proposition \ref{prop irifs} do not hold. However, if $x_1=x_2$, then all conclusions in that proposition are true. Here, $\Sigma$ is reducible.
\end{example}

Topological transitivity and point transitivity are equivalent in a classical dynamical system, which is not true for an IFS. In fact, in an IFS, topological transitivity always implies point transitivity \cite[Proposition 3.2]{ahmadi2023iterated}, but the converse is not necessarily the case \cite[Proposition 3.3]{ahmadi2023iterated}. The subsequent result directly follows from part \ref{tra prop} of the aforementioned proposition.

\begin{cor}
	Let $\mathfrak{I}$ and $\Sigma$ be as described in Proposition \ref{prop irifs}. Then, for $\mathfrak{I}_S=(S,\,\mathcal{F},\,\Sigma)$, topological transitivity and point transitivity are equivalent.
\end{cor}
\subsubsection{ Dynamical properties of the sub-shift versus  the IFS on the attractor}
The subsequent proposition reflects the dynamical properties of a sub-shift on the attractor when $\mathcal{F}$ preserves the attractor $S$.

\begin{prop}\label{prop IFS vs subshift}
	For an IFS where $\mathcal{F}$ is a contraction and preserves $S$, $\Sigma$ has property $P \in \mathfrak{P}=\{$Devaney chaos, mixing, weak mixing, totally transitive$\}$ if and only if $T_S:\Sigma\times S\rightarrow \Sigma\times S$ has property $P$.
\end{prop}
\begin{proof}
	The sufficiency is evident, so a proof for necessity will be given. Assume $\Sigma$ has Devaney chaos. By part \ref{item all s are trans} of Proposition \ref{prop irifs}, we only need to prove that periodic points are dense in
	$\Sigma \times S$. Taking an arbitrary open set $V \times W$ in $\Sigma \times S$ and $[v] \subseteq V$ for some word $v \in \mathcal{L}_{n}(\Sigma)$, choose a transitive point $t=vt_{n+1}\cdots$ in $[v]$ and find $b=vt_{n+1} \cdots t_{\ell}$ such that $f_{b}(S) \subseteq W$. Given the density of periodic points in $\Sigma$, there exists $u'^{\infty} \in [b]$. Set $u':=bq \in \mathcal{L}_{m}(\Sigma)$ and $\sigma:=u^\infty= \tau^{|b|}(u'^{\infty})=[qb]^\infty$. Consequently, $f_u(S)=f_{b}(f_q(S)) \subseteq W$. By another part of Proposition \ref{prop gifs}, the cycle $S_{\sigma} =\{x_0, f_{u_{1}}(x_0), \ldots ,f_{u_{1} \ldots u_{m-1}}(x_0)\}\subseteq S$ of length $m$ contains $x_{0}$ and thus $x_{0} \in W$ with $f_u(x_0)= x_{0}$.
	
	For the other properties in $\mathfrak{P}$, consider $[u_i]\times V_i$ as an open set in $\Sigma\times S$ for $i=1,\,2$ and let $t^{(i)}\in[u_i]$ be transitive. 
		Then, there is $u_i'$ such that $u_iu_i'$ is the initial of $t^{(i)}$ and $f_{u_iu_i'}(S)\subset V_i$.
		 Given $w$ such that $u_1u'_1\,w\,u_2u'_2\in\mathcal{L}(\Sigma)$, there is $s\in f_{u_1u_1'}(S)\subset V_1$ such that $f_{w\,u_2u'_2}(s)\in f_{u_2u'_2} (S)\subset V_2$. This implies $T^{l}_S(V_1)\cap V_2\neq\emptyset$ where $l=|w\,u_2u'_2|$. Now, if $\Sigma$ is mixing, there exists $M\in\N$ such that for $l\geq M$ and some $w\in\mathcal{L}_l(\Sigma)$, $u_1u'_1\,w\,u_2u'_2\in\mathcal{L}(\Sigma)$, leading to the conclusion that $T_S$ is mixing. Similar reasoning applies for the other properties.
\end{proof}

Recall that an IFS has Devaney chaos if it is transitive and possesses a dense collection of periodic points. Thus, the above proposition directly implies the following.

\begin{cor}\label{cor property}
	Assume $\mathcal{F}$, \IFS and $\mathfrak{P}$ are as in Proposition \ref{prop IFS vs subshift}. 
	If $\Sigma$ has property $P\in\mathfrak{P}$, then    $\mathfrak{I}_S=(S,\,\mathcal{F},\,\Sigma)$ has property $P$ as well.
	\end{cor}

The IFS with a trivial attractive set in Example \ref{eg 2egs} consists of $S=\{s_0\}$. Regardless of the choice of $\Sigma$, $\mathfrak{I}_S$ is mixing. Therefore, the converse of the above corollary is not generally valid.

\vskip .35cm
\noindent
{\bf Conclusion.}
 Let $\mathfrak{I}=(X,\,\mathcal{F},\,\Sigma)$, where $\Sigma$ is an arbitrary sub-shift, and $\mathcal{F}$ is a homeomorphism. We demonstrated that self-similarities can emerge in the attractor of $\mathfrak{I}$ when $\Sigma$ is coded, regardless of the choice of $\mathcal{F}$. However, the identification of self-similarities in other types of sub-shifts, particularly minimal sub-shifts, is not known. A key issue is to have $\mathcal{F}$ preserving the attractor. Example \ref{ex 12infty} highlights the potential challenge of this preservation when $\Sigma$ represents a minimal non-fixed cycle. In contrast, Proposition \ref{prop IFS vs subshift}, as a classical dynamical system, must have minimal sub-systems that result in IFS with minimal sub-shifts. The structure of these minimal sub-shifts evades our current understanding.
	\bibliographystyle{plain}
\bibliography{myreference}

\providecommand{\bysame}{\leavevmode\hbox to3em{\hrulefill}\thinspace}
\providecommand{\MR}{\relax\ifhmode\unskip\space\fi MR }
\providecommand{\MRhref}[2]{%
  \href{http://www.ams.org/mathscinet-getitem?mr=#1}{#2}
}
\providecommand{\href}[2]{#2}
\begin{thebibliography}{10}

\bibitem{ahmadi2023iterated}
Dawoud Ahmadi~Dastjerdi and Mahdi Aghaee, \emph{Iterated function systems over
  arbitrary shift spaces}, Khayyam Journal of Mathematics \textbf{9} (2023),
  no.~1, 127--143.

\bibitem{barnsley1985iterated}
Michael~F Barnsley and Stephen Demko, \emph{Iterated function systems and the
  global construction of fractals}, Proceedings of the Royal Society of London.
  A. Mathematical and Physical Sciences \textbf{399} (1985), no.~1817,
  243--275.

\bibitem{barnsley1988fractals}
Michael~F Barnsley, Robert~L Devaney, Benoit~B Mandelbrot, Heinz-Otto Peitgen,
  Dietmar Saupe, Richard~F Voss, and Richard~F Voss, \emph{Fractals in nature:
  from characterization to simulation}, Springer, 1988.

\bibitem{blanchard1986systemes}
Fran{\c{c}}ois Blanchard and Georges Hansel, \emph{Systemes cod{\'e}s},
  Theoretical Computer Science \textbf{44} (1986), 17--49.

\bibitem{bruin2022topological}
Henk Bruin, \emph{Topological and ergodic theory of symbolic dynamics}, vol.
  228, American Mathematical Society, 2022.

\bibitem{dastjerdi2023shift}
Dawoud~Ahmadi Dastjerdi and Mahdi Aghaee, \emph{Shift limits of a
  non-autonomous system}, Topology and its Applications \textbf{326} (2023),
  108394.

\bibitem{dastjerdi2019mixing}
Dawoud~Ahmadi Dastjerdi and Maliheh Dabbaghian~Amiri, \emph{Mixing coded
  systems}, Georgian Mathematical Journal \textbf{26} (2019), no.~4, 637--642.

\bibitem{dokukin2011cell}
Maxim~E Dokukin, Nataliia~V Guz, Ravi~M Gaikwad, Craig~D Woodworth, and Igor
  Sokolov, \emph{Cell surface as a fractal: normal and cancerous cervical cells
  demonstrate different fractal behavior of surface adhesion maps at the
  nanoscale}, Physical review letters \textbf{107} (2011), no.~2, 028101.

\bibitem{fiebig1992covers}
Doris Fiebig and Ulf-Rainer Fiebig, \emph{Covers for coded systems},
  Contemporary Mathematics \textbf{135} (1992), 139--179.

\bibitem{hutchinson1981fractals}
John~E Hutchinson, \emph{Fractals and self similarity}, Indiana University
  Mathematics Journal \textbf{30} (1981), no.~5, 713--747.

\bibitem{kolyada1996topological}
{Kolyada, Sergi{\u\i}} and Lubom\'ir Snoha, \emph{Topological entropy of
  nonautonomous dynamical systems}, Random and computational dynamics
  \textbf{4} (1996), no.~2, 205.

\bibitem{leggett2019motility}
Susan~E Leggett, Zachary~J Neronha, Dhananjay Bhaskar, Jea~Yun Sim,
  Theodora~Myrto Perdikari, and Ian~Y Wong, \emph{Motility-limited aggregation
  of mammary epithelial cells into fractal-like clusters}, Proceedings of the
  National Academy of Sciences \textbf{116} (2019), no.~35, 17298--17306.

\bibitem{lind2021introduction}
Douglas Lind, Douglas~A Lind, and Brian Marcus, \emph{An introduction to
  symbolic dynamics and coding}, Cambridge university press, 2021.

\bibitem{peitgen1986beauty}
Heinz-Otto Peitgen and Peter~H Richter, \emph{The beauty of fractals: images of
  complex dynamical systems}, Springer Science \& Business Media, 1986.

\bibitem{sobottka2022some}
Marcelo Sobottka, \emph{Some notes on the classification of shift spaces:
  Shifts of finite type; sofic shifts; and finitely defined shifts}, Bulletin
  of the Brazilian Mathematical Society, New Series \textbf{53} (2022), no.~3,
  981--1031.

\end{thebibliography}

\end{document}